\documentclass[12pt,a4paper]{article}
\usepackage[dvips]{graphicx}
\usepackage{amssymb,latexsym}
\usepackage{mathrsfs}
\usepackage{psfrag}
\usepackage[all]{xy}
\usepackage{color}
\usepackage{tabularx,ragged2e,booktabs,caption}
\usepackage{amsmath}
\usepackage{amssymb}
\usepackage{amscd}
\usepackage{enumerate}
\usepackage[english]{babel}

\newtheorem{proposition}{Proposition}[section]
\newtheorem{theorem}{Theorem}[section]
\newtheorem{observation}{Observation}[section]

\newtheorem{definition}{Definition}[section]
\newtheorem{remark}{Remark}[section]

\newenvironment{proof}
 {\begin{trivlist} \item[\hskip \labelsep {\bf Proof}\hspace*{3 mm}]}
 {\hfill$\Box$\end{trivlist}}

\begin{document}
\title{On the $FRS$-generic family of space cusps}
\author{Pouya Mehdipour \footnote{Faculty of federal university of Vi\c cosa (UFV) - Brazil.}\, and\,Mostafa Salarinoghabi\footnote{Post-doc at federal university of Vi\c cosa (UFV) - Brazil}}

\maketitle
\begin{abstract}
 We consider in this paper the $FRS$-deformations of a family of space curves with codimension $\leq 3$. Some geometric aspects of a space curve such as flattenings, vertices and twistings points has been studied.
\end{abstract}
\renewcommand{\thefootnote}{\fnsymbol{footnote}}
\footnote[0]{2010 Mathematics Subject classification 58K05, 53A04, 57R45.}
% % % % % % % % % % % % % % % % % % % % %
\section{Introduction}\label{sec:int}
The geometry of space curves and a classification of their singularities are well studied (see for example \cite{CarmenEster1,CarmenEster2,Uribe,gibhob}).
In this paper, we consider smooth ($C^\infty$) parametrised generic space curves $\gamma : \mathbb R \to \mathbb R^3$. 
The group of diffeomorphisms $\mathcal A$ is useful when we want to study singularities of $\gamma$ but if we also be interested in the geometry of $\gamma$, such as its flattenings, vertices and twistings, then we are allowed only Euclidean motions in the space  
as diffeomorphisms destroy the geometry of the curve 
(but preserve its singularities). 

%However, the problem of study the singularities of a submanifold of $\mathbb{R}^n$ together with its geometry is an open problem.
%
 
However, finding a theory that explains deformations of the singularity of a curve  
as well as the changes in the geometry of the curve  
is still an open problem. There are some results on this problem especially for plane curves (see for example \cite{NunoFabio,SenhaTari,wall,SalariTari}).

In \cite{SalariTari} has been proposed a method 
to study the geometry of deformations of singular plane curves. This method is so called {\it $FRS$-deformations of plane curves} (F for flat, R for round and S for singular). When the curve is regular they label the deformations {\it $FR$-deformations}.

More precisely, the definition of $FRS$-deformations is as following: 
\begin{definition}\label{def:main}
	Consider two germs of $m$-parameter deformations of the same plane curve: $\gamma_s, s \in ({\mathbb R}^m_1,0)$, and 
	$ \eta_u, u \in ({\mathbb R}^m_2,0).$ 
	Equip the base $({\mathbb R}^m_1,0)$ with a stratification $(S_1,0)$ such that if $s'$ and $s''$ are in the same stratum then the curves
	$\gamma_{s'}$ and $\gamma_{s''}$ satisfy the following properties:
	
	\begin{tabular}{rp{13cm}}
		{\rm (i)}& they are diffeomorphic;\\
		{\rm (ii)}& they have the same numbers of inflections and vertices;\\
		{\rm (iii)}& they have the same relative position of their singularities, points of self-intersection, inflections and vertices.
	\end{tabular}
	
	Also equip the base $({\mathbb R}^m_2,0)$ with a stratification $(S_2,0)$ with  properties  {\rm (i)--(iii)}. We say that the two deformations are {\em $FRS$-equivalent}
	if there is a stratified homeomorphism $k: ({\mathbb R}^m_1,(S_1,0)) \to  ({\mathbb R}^m_2,(S_2,0))$ such that all pairs of curves $\gamma_s$ and $\eta_{k(s)},$ $s \in ({\mathbb R}^m_1,0),$ satisfy properties {\rm (i)--(iii)}.
\end{definition} 

Basically, $FRS$-equivalence means the {\rm instantaneous} configurations of the curves  $\gamma_s$ and $\eta_{k(s)}$
(including their singularities and their geometries) are the same. Note that in the Definition \ref{def:main} the authors of \cite{SalariTari} has been used the notation $({\mathbb R}^m_i,0)$ instead of $({\mathbb R}^m,0)$ to point out which family is considered. 

Following the same idea as mentioned in \cite{SalariTari}, in this paper we deal with space curves at flattenings, vertices and twisting points (\S \ref{Sec:regular}) and with the 
singularities (\S \ref{Sec:singular}).

An $\mathcal A$-model for the space cusp curve is $(t^2,t^3,0)$ and a 
model of an $\mathcal A_e$-versal deformation 
of this singularity is $(t^2,t^3+ut,vt),(u,v)\in (\mathbb R^2,0)$.

The main result of this work is the Theorem \ref{theo:main}, where we study the bifurcation of a $FRS$-generic family of space cusps.
   
%========================================================
\section{Preliminaries}\label{sec:pre}
Let the function $f:J\to \mathbb R$ be smooth where $J$ is an open interval, 
we say that $f$ is singular at $t_0\in I$ if $f'(t_0)=0$. 
As we consider here local phenomena so we can assume $t_0=0$. 
The group $\mathcal R$ is the group of local changes of the variable in the source that 
fix $t=0$. The models for the local $\mathcal R$-singularities of functions are 
$\pm t^{k+1}$, $k\ge 1$, and these are known as $A_k$-singularities. 
The necessary and sufficient conditions for a function $f$ to have an $A_k$-singularity at $t=0$ are
$$
f'(0)=f''(0)=\ldots=f^{(k)}(0)=0,\, f^{(k+1)}(0)\ne0.
$$

A family of germs of functions $F:(\mathbb R\times \mathbb R^m,(0,0))\to \mathbb R$ with $F(t,s)=F_s(t)$ and $F_0(t)=f(t)$ 
is said to be an $\mathcal R^+$-{\it versal deformation} of the $A_k$-singularity of $f$ at $t=0$ if any other deformation $G$ of  
$f$ can be written as $G(t,u)=F(h(t,u),k(u))+c(u)$, 
for some germs of smooth functions $c$ and $h$ with $h(t,0)=t+O(t^2)$ 
and of a smooth map $k$. 
Let us denote the $\frac{\partial F}{\partial s_i}(t,0)$
by $\dot{F}_i(t)$. Then 
the family $F$ is an $\mathcal R^+$-versal deformation 
of $F_0=f,$ if and only 
if  
$
\mathcal E\langle f'\rangle+\mathbb R\langle1,\dot{F}_1,\ldots,\dot{F}_m\rangle
=\mathcal E,
$ 
where $\mathcal E$ is the ring of germs of smooth functions $(\mathbb R,0)\to \mathbb R$.
A model of an $\mathcal R^+$-versal deformation of the $A_k$-singularity $\pm t^{k+1}$ is 
$F(t,s)=\pm t^{k+1}+s_{k-1}t^{k-1}+\ldots+s_1t$.

Another important concept that we deal with here is the group of diffeomorphisms $\mathcal A$. Two
smooth maps $h$ and $g$ from $\mathbb R^n$ to $\mathbb R^m$  are said to be ${\cal A}$-{\it equivalent} if there are diffeomorphisms, $l: \mathbb R^n \to \mathbb R^n$ and $ k: \mathbb R^m\to \mathbb R^m$, such that $k\circ h=g\circ l.$  

A space curve $\gamma\in C^{\infty}(I,\mathbb{R}^3)$ is said to be generic if it is projection-generic (see \cite{David}) and satisfies the following conditions:
\begin{itemize}
	\item[(0)] If $\tau(t)=0$ for some $t$ then $\tau'(t)\neq 0$.
	\item[(1)] Assume that the secant line $l$ to $\gamma$ at two points $\gamma(t_i)$, $i = 1, 2$, is contained in the osculating planes $\mathcal{O}(t_i)$ for any $i = 1, 2$. Then $\tau(t_i)\neq 0$ for any $i = 1, 2$.
	\item[(2)] Let $l$ be a cross tangent to $\gamma$ at $\gamma(t_i)$ for any $i = 1,2$, tangent in $\gamma(t_1)$. Then $\gamma^{(4)}(t_1)\not\subset\mathcal{O}(t_1)$.
	\item[(3)] Let $l$ be a trisecant line to $\gamma$ at three points $\gamma(t_i)$, $i = 1, 2, 3$. If $l\subset\mathcal{O}(t_1)$, then $\tau(t_1)\ne 0$ and $l \not\subset \mathcal{O}(t_i)$ for any $i = 2,3$.
	\item[(4)] Let $l$ be a trisecant line to $\gamma$ at three points $\gamma(t_i)$, $i = 1, 2, 3$. If $l$ is contained in a bitangent plane $\pi,$ to $\gamma$ at two of these points, then $\pi$ does not osculate at any of the three points and $\gamma^{(3)}(t_i)$ is not contained in the bitangent plane.
	\item[(5)] Let $l$ be a quadrisecant line to $\gamma$ at four points $\gamma(t_i)$, $i = 1, 2, 3, 4$. Then, $l\not\subset\mathcal{O}(t_i)$ and $l||\gamma^{(3)}(t_i)$ at most in two points.
\end{itemize} 
Let $\gamma: \mathbb{R} \to \mathbb{R}^3$ be a generic space curve parametrized by arc-length parametrization.
The function $d_{\gamma} : \mathbb R\times \mathbb R^3 \to \mathbb R$ with $d_{\gamma}(t,x)=\frac{1}{2}\langle \gamma(t)-x,\gamma(t)-x\rangle$, is called distance squared function of $\gamma$. 

We say that the sphere $S^2(a,r)$, with center $a$ and radius $r$, has contact of order $k$ with $\gamma$ at $t_0$ if and only if 
\begin{align*}
d_{\gamma}(t_0,a)=r^2, \frac{\partial d_{\gamma}}{\partial t}(t_0,a)=\cdots=\frac{\partial^k d_{\gamma}}{\partial t^k}(t_0,a)=0 \,\text{and}\,\, \frac{\partial^{k+1} d_{\gamma}}{\partial t^{k+1}}(t_0,a)\ne 0. 
\end{align*}

Suppose that $\{T(t),N(t),B(t)\}$ be the Frenet frame at $t$ where $T(t)=\gamma'(t)$, $N(t)$ is the normal vector and $B(t)=T(t)\times N(t)$ is the binormal vector of $\gamma$ at $t$. The curvature and torsion of $\gamma$ at $t$ are denoted by $\kappa(t)$ and $\tau(t)$ respectively. 
%Therefore the Frenet equations are
%\begin{align*}
%T'(t)&= \kappa(t)N(t),\\
%N'(t)&= -\kappa(t)T(t)+\tau(t) B(t),\\
%B'(t) &= -\tau(t) N(t).
%\end{align*}
We define the concept of generalized evolute of $\gamma$ as follow.
\begin{definition}\label{Def:genEvo}
For the space curve $\gamma$, the centers of all osculating hyperspheres of $\gamma$ form a smooth curve $c_{\gamma} : \mathbb R \to \mathbb R^3$ such that $c_{\gamma}(t)=\gamma(t)+\mu_1(t) N(t)+\mu_2(t)B(t)$. The curve $c_{\gamma}$ is called \textit{generalized evolute} or \textit{curve of spherical curvature centers} of $\gamma$ or \textit{focal curve} of $\gamma$.
\end{definition}
The coefficients $\mu_i$ in Definition \ref{Def:genEvo} are called focal curvatures and are rational functions of the curvature and torsion of $\gamma$ and their derivatives. Also the osculating hypersphere at $t$ has the radius equal $\sqrt{\mu_1^2+\mu_2^2}$ (see \cite{CarmenEster2} for more details).

According to \cite{CarmenEster2, Uribe} one can show that the curve $c_{\gamma}$ is singular if and only if the osculating hypersphere $S^2$ at each point $t$ of $\gamma$ has at least 4-point contact with $\gamma$. This means that the distance squared function $d_{\gamma}$ has $A_4$-singularity.
\begin{observation}
It is easy to prove that for a space curve $\gamma$, the generalized evolute of $\gamma$ is the curve
\begin{align}
c_{\gamma}(t)=\gamma(t)+\frac{1}{\kappa(t)} N(t)-\frac{\kappa '(t)}{\kappa^2(t)\tau(t)} B(t).
\end{align}
\end{observation}
\begin{definition}\label{Def:vertex}
A point $\gamma(t_0)$ where $c_{\gamma}'(t_0)=0$ and $\tau(t_0)\ne 0$ is called a \textit{vertex} of $\gamma$. 
\end{definition}  
In \cite{Uribe} it is shown that a point $p=\gamma(t_0)$ with $\tau(t_0)\ne 0$ is a vertex of $\gamma$ if and only if $\mu'_2+\mu_1 \tau =0$ at $t=t_0$. A general study about vertices and focal curvatures of space curves is given in \cite{Uribe}.

Another geometric phenomenon in study of space curves is \textit{flattening point}. A flattening point for the space curve $\gamma$ is a point $p=\gamma(t_0)$ such that $\tau(t_0)=0.$ 
Equivalently, $\gamma(t_0)$ is a flattening point if and only if $$\det(\gamma'(t_0),\gamma''(t_0),\gamma'''(t_0))=0.$$ These points captured by an $A_3$-singularity of height function along normal direction at $t_0$.
A \textit{bi-flattening} point of $\gamma$ is a point $p=\gamma(t_0)$ such that $\tau(t_0)=\tau'(t_0)=0$, we recall that a bi-flattening point is not generic.

A helix is a curve $\alpha:\mathbb R \to \mathbb R^3$ such that its tangent vector forms a constant angle with a given direction $v$ at $\mathbb R^3$.
\begin{proposition} [\cite{CarmenEster1}]
A curve $\alpha : \mathbb R \to \mathbb R^3$ is a helix if and only if the function $\det(\alpha''(t),\alpha'''(t),\alpha^{(4)}(t))$ is identically zero, where $\alpha^{(i)}$ represents the ith derivative of $\alpha$ with respects to arc-length.
\end{proposition}
\begin{definition}\label{Def:Twisting}
If $\gamma: \mathbb R \to \mathbb R^3$ is parametrized by its arc-length $t$, then a point $p=\gamma(t_0)$ is a twisting (or Darboux vertex) of $\gamma$ if it is a critical point of function $\frac{\tau}{\kappa}$. This means that 
\[
\tau'(t_0)\kappa(t_0)-\kappa'(t_0)\tau(t_0)=0.
\]
\end{definition}
\begin{proposition}[\cite{CarmenEster1}]
Given a space curve $\gamma$ parametrized by arc-length in $\mathbb R^3$, a point $t_0$ is a twisting of $\gamma$ if and only if there exists some helix whose order of contact with $\gamma$ at $t_0$ is at least 4.
\end{proposition}
%It is shown in \cite{Tenen} that at a twisting of a given space curve $\gamma$ in $\mathbb R^3$ we have $$\det(\gamma''(t_0),\gamma'''(t_0),\gamma^{(4)}(t_0))=0.$$ 
% % % % % % % % % % % % % % % % % % % % % % % % % % %
\section{Regular space curves}\label{Sec:regular}
We devote this section to study the $FR$-deformations of a regular space curve.

Suppose that $\gamma$ has a flattening at $t_0$. 
To study the deformations of the flattening, we use the family of height function $\tilde H$ on $\gamma_s$. If it is an 
$\mathcal R^+$-versal deformation, then we have a well studied model of the deformation of the flattening.

Without loose of generality, one can assume that $t_0=0$, $T(0)=(1,0,0)$, $N(0)=(0,1,0)$ and $B(0)=(0,0,1)$. Applying the Serret-Frenet formulae and using the Taylor expansion of $\gamma$ we obtain
\begin{align}\label{gamma}
\gamma(t)=\left( \begin{array}{c}
t+a_3t^3+a_4t^4+O(t^5)\\
b_2 t^2+b_3 t^3+b_4 t^4+ O(t^5)\\
c_3 t^3+c_4 t^4+ O(t^5)  \end{array} \right),
\end{align}
where $b_2>0$. 

Using the expression $\tau(t)=\frac{\langle \gamma'\times\gamma'' , \gamma'''\rangle}{\langle \gamma'\times\gamma'',\gamma'\times\gamma''\rangle}$ one can easily obtain $\tau(0)=b_2c_3$. This means that at a flattening point of $\gamma$ we have $c_3=0$. 

\begin{proposition}\label{Prop:flat}
If the space curve $\gamma$ has a flattening point at $p$ then the generalized evolute goes to infinity asymptotically along the binormal direction of $\gamma$ at $p$ and can be modeled in some coordinate system by $xyz^3=1$ (see Figure \ref{fig: evflatt}).   
\end{proposition}
\begin{proof}
At a flattening point $\tau=0$. We parametrize the space curve $\gamma$ by \eqref{gamma} such that $a_3<0$, $b_2>0$ and $c_3=0$. We have
 $$\kappa(t)=\frac{2b_2+6b_3t+(-6a_3b_2+12b_4)t^2+O(t^3)}{(1+(3a_3+2b_2^2)t^2+O(t^3))^{\frac{3}{2}}}$$ and $$\tau(t)=\frac{12b_2c_4t+18c_4b_3 t^2+O(t^3)}{b_2^2+6b_2b_3 t +O(t^2)}.$$ Also the normal and binormal vectors are $$N(t)=(-2b_2t+O(t^2),1-2b_2^2t^2+O(t^3),\frac{6c_4}{b_2}t^2+O(t^3))$$ and $$B(t)=(8c_4t^3+O(t^4),\frac{-6c_4}{b_2}t^2+O(t^3), 1-\frac{18c_4^2}{b_2^2}t^4+O(t^5)).$$ 

The generalized evolute of $\gamma$ is parametrized by 
\begin{align}\label{evolute flat}
c_{\gamma}(t) & = \gamma(t)+\frac{1}{\kappa(t)} N(t)-\frac{\kappa '(t)}{\kappa^2(t)\tau(t)} B(t) \nonumber \\
              & = \left(\frac{b_3}{2b_2} t^2+O(t^3), \frac{1}{2b_2} -\frac{3b_3}{4b_2^2} t+O(t^2), \frac{b_3}{-8c_4b_2 t}(1+O(t))\right).
\end{align}
Clearly, the expression \eqref{evolute flat} in an appropriate coordinate system can be written in the form  $xyz^3=1$. 
\end{proof}

\begin{remark}
Note that height function has an $\mathcal{A}_3$-singularity at $p=\gamma(0)$ if and only if $p$ is a flattening point of $\gamma$. In fact we have
 
%\begin{equation}
\begin{tabular}{ccl}
$H'(0)=0$ & $\iff$ &   $w=\eta_1 N(0) + \eta_2 B(0),$ \\
$H'(0)=H''(0)=0$ & $\iff$ &  $w=\eta_1 N(0) + \eta_2 B(0),$\\
 & &			$\eta_1=0.$ \\

$H'(0)=H''(0)=H'''(0)=0$ & $\iff$ &  $w=\eta_1 N(0) + \eta_2 B(0),$\\
 & &$\eta_1=0,$\\
& & $\tau(0)=0.$\\
$H^{(4)}(0)=\kappa(0)\tau'(0)\ne 0.$ & & 
\end{tabular}
%\end{equation}

%When the family of height functions is $\mathcal R^+$-versal, the resulting family of torsions $\tau(t,s)$ is an $\mathcal R$-versal family.
\end{remark}
\begin{definition}\label{def:FR flattening}
 We say that a deformation of a regular space curve $\gamma$ at a flattening point is $FR$-generic if the associated family of height functions is an $\mathcal{R}^+$-versal deformation of a singularity of the height function on $\gamma$ along its bi-normal direction.
 \end{definition}
As the flattening points are stable therefore we have the following Theorem about $FR$-model at a flattening point.
 \begin{theorem}
Any $FR$-generic stable family of curves at a flattening point is $FR$-equivalent to the model $(t,t^2,t^4)$.
 \end{theorem}
\begin{proof}
 The above calculations depend only on the fact that the space curve has a flattening point at the origin (i.e. $c_3=0$ and $c_4\ne 0$). Using the Definition \ref{def:FR flattening} we conclude that the curve $(t,t^2,t^4)$ satisfies these conditions and has $\mathcal{A}$-codimention 0.  
 
\end{proof}
\begin{figure}[h]
\centering
\includegraphics[width=0.4\textwidth]{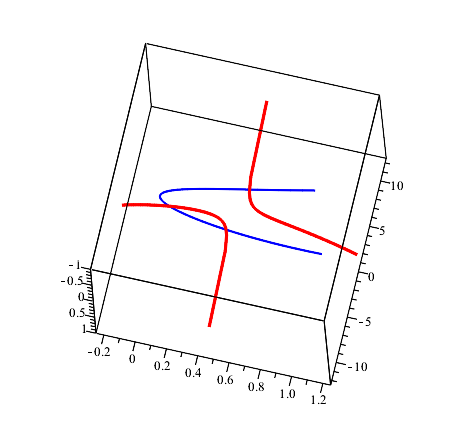}
\caption{The generalized evolute (red curve) at a flattening point. Evolute goes to infinity along the binormal direction.} \label{fig: evflatt}
\end{figure}

%\textbf{Vertex}

 Away from flattening points of $\gamma$, the deformations of the curve at a vertex point $t_0=0$ of finite order can be studied using the family of distance squared functions $\tilde{D}(t,q,s)$. If the family $\tilde{D}$ on $\gamma_s$ with $\gamma_0$ as in \eqref{gamma} is an $\mathcal R^+$-versal deformation of the singularity of $\tilde{D}(t,0,0)$ at origin, then we say the family $\gamma_s$ is $FR$-generic.
 
  By a direct calculation one can obtain the parametrization of the generalized evolute away from flattening points. In this case $c_3\ne 0$ and a vertex happens at $t_0=0$ if $b_2^3c_3+2a_3b_2c_3+b_3c_4-b_4c_3=0.$ Equivalently we have $\kappa''(0)=\frac{2\kappa'(0)^2}{\kappa(0)}+\frac{\kappa'(0)\tau'(0)}{\tau(0)}+\kappa(0)\tau^2(0)$. To be more precise, the generalized evolute at a vertex point has the following parametrization:
  \begin{align}\label{general ev}
    c_{\gamma}(t)=\left(\bar{a}_4 t^4+O(t^5), \bar{b}_0+\bar{b}_3 t^3 +O(t^4), \bar{c}_0 +\bar{c}_2 t^2 +O(t^3)\right),
    \end{align}  
   where
   \begin{align*}
   \bar{a}_4 = & \frac{3(8b_2^2b_3c_3-3a_3b_3c_3+10a_4b_2c_3+5b_3c_5-5b_5c_3)}{2c_3b_2}, \\
   \bar{b}_0 = & \frac{1}{2b_2},\\
   \bar{b}_3 = & \frac{-34b_2^2b_3c_3+9a_3b_3c_3-40a_4b_2c_3-20b_3c_5+20b_5c_3}{2c_3b_2^2}  ,\\
   \bar{c}_0 = & \frac{-b_3}{2b_2c_3},\\
   \bar{c}_2 = & \frac{18b_2^2b_3c_3-3a_3b_3c_3+20a_4b_2c_3+10b_3c_5-10b_5c_3}{2c_3^2b_2}.
   \end{align*}
  %\begin{align}\label{general ev}
  %c_{\gamma}(t)=\left(\bar{a}_3 t^3+O(t^4), \bar{b}_0+\bar{b}_2 t^2 +O(t^3), \bar{c}_0 +\bar{c}_1 t +O(t^2)\right),
  %\end{align} 
  %where $\bar{a}_3=\frac{4b_2^3c_3^2+8a_3b_2c_3^2+4b_3c_3c_4-4b_4c_3^2}{2b_2c_3^2}$, $\bar{b}_0=\frac{1}{2b_2}$, $\bar{b}_2=\frac{-(6b_2^3c_3+12a_3b_2c_3+6b_3c_4-6b_4c_3)}{2b_2^2c_3}$, $\bar{c}_0= \frac{-b_3}{2b_2c_3}$ and $\bar{c}_1=\frac{2(b_2^3c_3+2a_3b_2c_3+b_3c_4-b_4c_3)}{b_2c_3^2}$.
  
 % Note that the generalized evolute \eqref{general ev} passes through the point $(0,\frac{1}{\kappa(0)},\frac{-\kappa'(0)}{\kappa^2(0)\tau(0)})$.
 %The terms $\frac{1}{\kappa}$ and $\frac{-\kappa'}{\kappa^2\tau}$ are called focal curvatures of $\gamma$.  
 
%If the space curve $\gamma$ has a vertex at the origin then we have $b_4c_3=b_2^2c_3+2a_3b_2c_3+b_3c_4$, substituting this in $\bar{a}_3$ and $\bar{b}_2$ we obtain $\bar{a}_3=\bar{b}_2=0$. Therefore the parametrization of the generalized evolute \eqref{general ev} becomes as following.

 The generalized evolute at a vertex point $p$ of $\gamma$ is given in Figure \ref{fig: evVertex}.
\begin{figure}[h]
\centering
\includegraphics[width=0.4\textwidth]{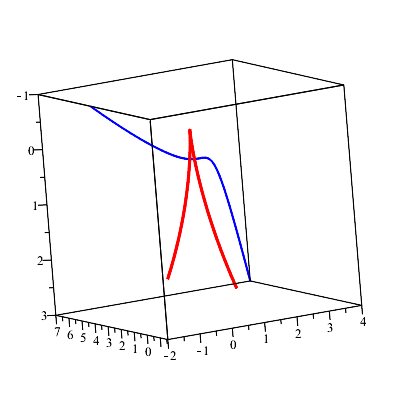}
\caption{The generalized evolute (red curve) at a vertex point. } \label{fig: evVertex}
\end{figure}

As we mentioned in Definition \ref{Def:Twisting} another interesting geometric phenomenon in study of space curves are twisting points. The twisting points are flattening points of tangent indicatrix of $\gamma$ i.e. the map $\gamma_T : \mathbb R \to \mathbb R^3$ with $\gamma_T(t)=\gamma'(t)$ (see Definition \ref{Def:Twisting}). More precisely, 
if $H(t) = \langle T(t),w\rangle,$ be the height function of $\gamma_T$ along the vector $w$, then we have
	
	\begin{tabular}{lcl}
		$H'(0)= 0$ & $\iff$ & $w=\lambda_1 T(0)+\lambda_2 B(0),$\\
		&&\\
		$H'(0)=H''(0)=0$ & $\iff$ & $w=\lambda_1 T(0)+\lambda_2 B(0),$ \\
		& & $\lambda_1 = \frac{\tau(0)}{\kappa(0)} \lambda_2,$\\
		&& \\
		$H'(0)=H''(0)=H'''(0)=0$ & $\iff$ & $w=\lambda_1 T(0)+\lambda_2 B(0),$\\
		& & $\lambda_1 = \frac{\tau(0)}{\kappa(0)} \lambda_2,$\\
		& & $\kappa(0)\tau'(0)-\kappa'(0)\tau(0)=0.$ 
	\end{tabular}

Therefore the twisting points of $\gamma$ capture by considering the singularities of height function of tangent indicatrix of $\gamma$ along a rectifynig vector at $\gamma(t)$. 

We may say a family of space curves $\gamma_s(t)$ at a twisting point $p=\gamma(t_0)$ is $FR$-generic if the family of height functions of its indicatrix be $\mathcal R^+$-versal.

\begin{proposition}
If the space curve $\gamma$ has a twisting at $t=0$ then its generalized evolute has a twisting at $t=0$ as well.
\end{proposition}  
\begin{proof}
For the space curve $\gamma$ given in \eqref{gamma} it is not difficult to show that at a twisting point the generalized evolute has the following parametrization.
\[
c_{\gamma}(t)=(-\frac{\delta}{6b_2^2} t^3+O(t^4),\frac{\delta}{4b_2^3} t^2+O(t^3),-\frac{b_3}{2b_2c_3}-\frac{\delta}{b_2^2c_3} t +O(t^2)),
\] 
where $\delta=4b_2^4+12b_2b_4-27b_3^2$ is not equal zero generically. We denote the curvature and torsion of the generalized evolute by $\kappa_{c}$ and $\tau_{c}$ respectively. We have
\begin{align*}
\kappa_c(t) =  \frac{18 b_2 c_3^2}{|\delta|}+O(t), \quad  
\tau_c(t) =  - \frac{12c_3 b_2^3}{\delta} + O(t).
\end{align*}
Then by using the Maple package one can obtain that 
$$\kappa'_c\tau_c-\tau'_c\kappa_c=\pm \frac{216 b_2^2c_3^2 \tilde{\delta}}{\delta^2} t + O(t^2),$$
where \[\tilde{\delta}=12b_2^4c_3-20b_2^2c_5+48b_2b_4c_3+27b_3^2c_3+27c_3^3.\]
Therefore the generalized evolute has a twisting at $t=0$.
\begin{figure}[h]
\centering
\includegraphics[width=0.4\textwidth]{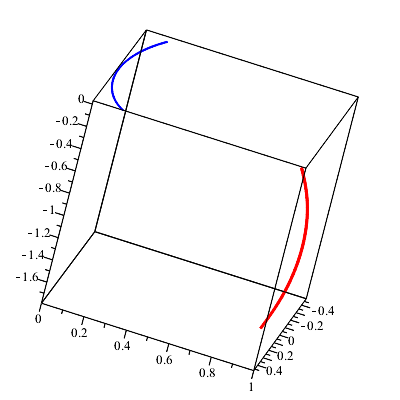}
\caption{The generalized evolute (red curve) at a twisting point. } \label{fig: evTwist}
\end{figure}
\end{proof}
% % % % % % % % % % % % % % % % % % % % % % % % % % % % % % % % % % % %
\section{Singular space curves}\label{Sec:singular}
 In \cite{gibhob} C. G. Gibson and C. A. Hobbs has classified simple singularities of space curves. In this paper we will study $FRS$-deformations of singular space curves with $\mathcal{A}$-codimension $\leq 3$ where it is shown in \cite{gibhob} that there is just one orbit $\mathcal{A}$-equivalent to $(t^2,t^3,0)$ with $\mathcal{A}$-codimension 2 for such space curves. We call a space curve $\mathcal{A}$-equivalent to $(t^2,t^3,0)$ a \textit{space cusp}. 
 
 Next proposition is an extension of Proposition 2.1 of \cite{SalariTari} which is about non-versality of family of distance squared functions of a space cusp.
 \begin{proposition}
 Let $\gamma_s$ be any $m$-parameter family of space cusps $\gamma$. Then the big family of distance squared functions $\tilde{D} : (\mathbb{R}\times\mathbb{R}^m\times\mathbb{R}^3, (0,0,0)) \to (\mathbb{R},0)$ on $\gamma$ with $\tilde{D}(t,s,a)=\langle \gamma_s(t)-a,\gamma_s(t)-a\rangle$ is never an $\mathbb{R}^+$-versal deformation of the singularity of $D_0(t)=\tilde{D}(t,0,0)$ at $t=0$.
 \end{proposition}  
 \begin{proof}
 By similar arguments to those in Proposition 2.1 of \cite{SalariTari} one can get the proof.
 
 \end{proof}
 The space cusp $\gamma$ has $\mathcal{A}$-codimension 2. Therefore following the same approach as in \cite{SalariTari}, we consider a stratification of $k$-jet space $J^k(1,3)$. For local strata in $J^k(1,3)$ let $\gamma(t)=(\alpha(t),\beta(t),\theta(t)),$ therefore we can take $\gamma(t_0)$ to be the origin and consider the Taylor expansion map $j^k\phi_{\gamma} : J \to J^k(1,3)$, where $J$ is an open interval of $\mathbb{R}$, with
 \begin{align*}
 j^k\phi_{\gamma}(t_0)=\left(\alpha'(t_0),\frac{1}{2}\alpha''(t_0), \dots ; \beta'(t_0),\frac{1}{2}\beta''(t_0), \dots ; \theta'(t_0),\frac{1}{2}\theta''(t_0), \dots\right).
 \end{align*}
 Let us denote the coordinates in $J^k(1,3)$ by $(a_1,a_2,\dots ; b_1,b_2,\dots ; c_1,c_2,\dots),$  then we have the following local strata in $J^k(1,3)$:
 
 \medskip
 \begin{tabular}{lcl}
 Cusp $(C)$&:& 
 $
  \begin{array}{l}
 a_1=b_1=c_1=0,
 \end{array}
 $
 \\
 &
 \\
 Flattening $(F)$&:& 
 $
 
 \begin{array}{l}
 a_1(b_2c_3-b_3c_2)+b_1(a_3c_2-a_2c_3)+c_1(a_2b_3-a_3b_2) = 0,
 \end{array}
 $\\
 &
 \\
 Vertex $(V)$&:&
 $
 \begin{array}{l}
 a_1 \xi_1 +b_1 \xi_2 +c_1 \xi_3 + V_2(a_1,b_1,c_1)+V_3(a_1,b_1,c_1) = 0,
 \end{array}
 $\\
  &
  \\
  Twisting $(\mathcal{T})$ &:&
  $
  \begin{array}{l}
  \mathcal{T}_6(a_1,b_1,c_1) + \mathcal{T}_7(a_1,b_1,c_1) = 0,
  \end{array}
  $ 
 \end{tabular}
 \medskip
 
 where 
 \begin{align*}
\xi_1 = & 24(b_2c_3-b_3c_2)(a_2^2+b_2^2+c_2^2), \\
\xi_2 = & 24(a_3c_2-a_2c_3)(a_2^2+b_2^2+c_2^2), \\
\xi_3 = & 24(a_2b_3-a_3b_2)(a_2^2+b_2^2+c_2^2),
 \end{align*}
 and $\mathcal{T}_i(a_1,b_1,c_1)$ and $V_i(a_1,b_1,c_1)$ denote  monomials of degree $i$ with respect to $a_1$, $b_1$ and $c_1$. 

Precisely, the above strata of $J^k(1,3)$ is captured by considering the zero set of the determinant $\det(\gamma'(0),\gamma''(0),\gamma'''(0))$ for flattening stratum $(F)$ and $A_4$-singularity of the distance squared function of $\gamma$ for vertex stratum ($V$). For the twisting stratum one can observe that the numerator of $\tau'(0)\kappa(0)-\kappa'(0)\tau(0)$ has the following expression:
 {\footnotesize 
\begin{align*}
\mathcal{T}=  36 & \big(a_1^2(b_2^2+c_2^2)-2a_1b_1a_2b_2-2a_1c_1a_2c_2+b_1^2(a_2^2+c_2^2)+c_1^2(a_2^2+b_2^2)-2b_1c_1b_2c_2\big)^2\\
& \big(a_1a_2+b_1b_2+c_1c_2\big)\big(a_1(b_2c_3-b_3c_2)+b_1(a_3c_2-a_2c_3)+c_1(a_2b_3-a_3b_2)\big)+\mathcal{T}_7(a_1,b_1,c_1).
\end{align*}}
At a cusp point we have $a_2^2+b_2^2+c_2^2\ne 0$. At such points, the strata $F$, $V$ and $\mathcal{T}$ are codimension 2 sub-varieties of $J^k(1,3)$ and contain codimension 3 sub-variety $C$. The sub-varieties $F$ and $V$ are tangent. The sub-variety $\mathcal T=\tilde{\mathcal T}_1\cup\tilde{\mathcal T}_2\cup\tilde{\mathcal T}_3,$ with $\tilde{\mathcal T}_1$, $\tilde{\mathcal T}_2$ and $\tilde{\mathcal T}_3$ smooth sub-varieties intersecting transversally along $C$. The tangent cone of $\tilde{\mathcal T}_1$ along $C$ is  $a_1a_2+b_1b_2+c_1c_2=0$ and that of $\tilde{\mathcal T}_2$ is $a_1(b_2c_3-b_3c_2)+b_1(a_3c_2-a_2c_3)+c_1(a_2b_3-a_3b_2)=0$, so $\tilde{\mathcal{T}}_2$ is tangent to $F$ and $V$ along $C$ and $\tilde{\mathcal{T}}_1$ and $\tilde{\mathcal{T}}_3$ are transverse to $F$ and $V$.

 As mentioned before, we have an $\mathcal{A}$-codimension 2 phenomenon. For a 2-parameter family of space cusps we need to have $\mathcal{A}_e$-versal deformation. Therefore we can take local parametrisation of the family of $\gamma_s$ in the form 
 \[
 \gamma_s(t)= (a_2 t^2+a_3t^3+\cdots, \bar{b}_1(s)t+\bar{b}_2(s)t^2+\cdots,\bar{c}_1(s)t+\cdots),
 \]
with $\bar{b}_1(s)=s_1$, $\bar{c}_1(s)=s_2$, $a_2 \ne 0$, $\bar{b}_2(0)=0$, $\bar{c}_2(0)=0$, $\bar{b}_3(0)=b_3\ne 0$, $\bar{b}_i(0)=b_i$ and $\bar{c}_j(0)=c_j$ for $(i\geq 4\,\,,\,\,j\geq 3)$. 
%\begin{theorem}\label{theo:transvers}
%Let $\gamma_s$ be a 2-parameter $\mathcal{A}_e$-versal deformation of the space cusp singularity which is $\mathcal{A}$-equivalent to $(t^2,t^3,0)$. Then the family of Taylor maps $j^k\Phi : (\mathbb{R}\times\mathbb{R}^2,(0,0)) \to\, J^k(1,3)$ given by $j^k\Phi(t,s_1,s_2)=\, j^k\phi_{\gamma_s}(t)$, is transverse to cusp stratum ($C$) and consequently to stratification $(S)$ of $J^k(1,3)$.
%\end{theorem}
%\begin{proof}
%The transversality to cusp stratum follows from Mather's theorem on $\mathcal{A}_e$-versal family of map germs.
%As the stratum $(C)$ has codimension 3 therefore transversality to $(C)$ implies transversality to whole stratification as well. 
%\end{proof}

The bifurcation set in the family $\gamma_s$ consists  of the projection to the $s$-parameter space $\mathbb R^2$ 
of the pull-back of the stratification $S$ in $J^k(1,3)$ by $j^k\Phi$. 
We can compute the initial parts of the parametrisations or equations of the curves forming the bifurcation set. 

The bifurcation set of the family $\gamma_s$ consists of the following curves:
\vspace{0.5cm}

\begin{tabular}{cl}

\vspace{0.3cm}
$C$: & $s_1=s_2=0$,\\
\vspace{0.3cm}
$F$: & $s_2=\displaystyle{\frac{c_3}{b_3}s_1}$,\\
\vspace{0.4cm}
$V$: & $24a_2^3 (b_3 s_2 -c_3 s_1) +48a_2(b_3s_1+c_3s_2)(b_3s_2-c_3s_1)+O(3)=0$,\\
\vspace{0.4cm}
$\mathcal{T}$: & $\displaystyle{\frac{6a_2}{(s_1^2+s_2^2)^2}\bigg( (4a_2b_4-9a_3b_3)s_2+(9a_3c_3-4a_2c_4)s_1\bigg)}=0$,\\

\end{tabular}

Note that if $(s_1,s_2)\ne (0,0),$ then one can consider the numerator of $\mathcal{T}$ as our desire stratum. Moreover using the implicit function theorem one can get the parametrization of vertex stratum ($V$) given by 
\[
(s_1,s_2) = (t, \frac{c_3}{b_3} t + \frac{a_3(b_3^2+c_3^2)(b_3c_4-b_4c_3)}{a_2^3b_3^4} t^3 + O(t^4)).
\]
Also it is evident that the strata $F$ and $V$ are tangent and $\mathcal{T}$ is transverse to them if $b_4c_3-b_3c_4 \ne 0$. Note that it is not difficult to observe that there is no multi-local stratum in the bifurcation of family of space cusps. 
 
We have the following definition on $FRS$-generic family of a space cusp.
\begin{definition}\label{def:FRS}
Let $\gamma_0=(a_2 t^2+a_3t^3+\cdots,b_3t^3+b_4t^4+\cdots,c_3t^3+\cdots)$ be a space cusp with $a_2b_3\ne 0$ and suppose that $b_4c_3-b_3c_4 \ne 0$. A 2-parameter deformation of $\gamma_0$ is said to be $FRS$-generic if it is an $\mathcal{A}_e$-versal deformation of the space cusp $\gamma_0$. 
\end{definition}
The $FRS$-stratification of an $FRS$-generic deformation of a space cusp is given in Figure \ref{fig:FRS1}. 

\begin{figure}[h]
\centering
\includegraphics[width=0.4\textwidth]{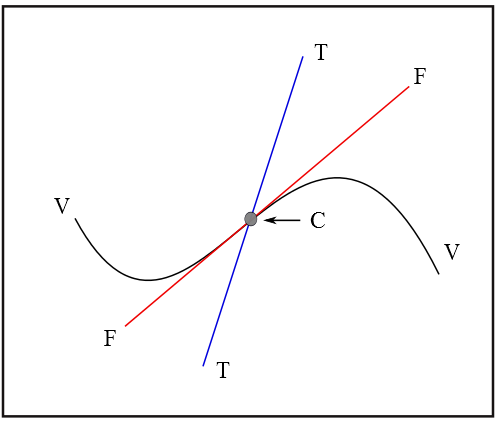}
\caption{$FRS$-stratification of a family of space cusps. } \label{fig:FRS1}
\end{figure}

\begin{theorem}\label{theo:main}
Any $FRS$-generic 2-parameter family of a space cusp is $FRS$-equivalent to $FRS$-model family $$G(t,u,v)=(t^2+t^3+t^4+O(t^5),ut+t^3+t^4+O(t^5),vt+t^3-t^4+O(t^5)).$$
\end{theorem}
\begin{proof}
According to Definition \ref{def:FRS} all calculations in this section depend only on the fact that the 2-parameter family of space cusp is $FRS$-generic. A similar calculations hold for $G$.   
\end{proof}

%\begin{conj}
%	There is an accumulation of 2 flattenings, 2 vertices and 2 twistings at space cusp singularity. In other words, the maximum number of flattening, vertex and twisting in the $FRS$-bifurcation of a family of a space cusp is 2.
%\end{conj}
% % % % % % % % % % % % % % % % % % % % % % % % % % % % % % % % % % % %
{\small

P.M. : Federal university of Vi\c cosa, Brazil. \\
pouya@ufv.br

M. S. : Federal university of Vi\c cosa, Brazil. \\
84.mostafa@gmail.com
}
\end{document}